\documentclass[11pt,a4paper]{article}
\usepackage[pagewise]{lineno}
\usepackage{amsfonts}
\usepackage{epic}
\usepackage{graphicx}
\usepackage{flafter}
\usepackage{geometry}
\usepackage[centertags]{amsmath}
\usepackage{amsfonts}
\usepackage{amssymb}
\usepackage{dsfont}
\usepackage{amsthm}
\usepackage{bbm}
\usepackage{mathrsfs}
\usepackage{epsfig}
\usepackage{color}
\usepackage{amsmath}
\numberwithin{equation}{section}
\numberwithin{figure}{section}

\date{}
\DeclareMathAlphabet{\mathds}{OT1}{cmss}{m}{sl}
\newtheorem{theorem}{Theorem}[section]
\newtheorem{lemma}{Lemma}[section]

\begin{document}

\title{\bf{Monochromatic Regular Polygons \\in finitely colored $\mathbb{Z}_2*\mathbb{Z}_2*\mathbb{Z}_2$}}
\author{Hui Xu\ \& \ Enhui Shi} 
\date{\today}

\maketitle
\begin{abstract}
We show that for any finite coloring of the group $\mathbb{Z}_2*\mathbb{Z}_2*\mathbb{Z}_2$  and for any positive integer $k$,
there always exists a monochromatic regular $k$-gon in $\mathbb{Z}_2*\mathbb{Z}_2*\mathbb{Z}_2$ with respect to the word length metric induced by the standard generating set; the edge length of which is estimated.
\end{abstract}


\section{Introduction}

Van der Waerden's celebrated theorem states that for any finite coloring of the integer group $\mathbb Z$, there always exist monochromatic arithmetic progressions of arbitrarily length (\cite{W}). An extension of van der Waerden's theorem to finite colorings on $\mathbb Z^d$ is obtained by Gallai (\cite{R}).
Furstenberg and Weiss gave a new proof of above results, based on methods of topological dynamics---they proved the well-known
MBR theorem for commuting homeomorphisms (\cite{FW1}). The MBR theorem for nilpotent group actions and its combinatorial corollaries are established by
Bergelson and Leibman (\cite{BL1, L1}). \\

In this paper, we consider the Ramsey properties for finite colorings of nonamenable groups. We mainly concern the free
product $\mathbb{Z}_2*\mathbb{Z}_2*\mathbb{Z}_2$ of $2$-order groups, since its Cayley graph is very simple and suitable for analysing by combinatorial
methods. Reall that for any finitely generated group $G$ with a fixed generator set $S$, there always exists a word length metric on
$G$ induced by $S$. If the group $G$ is finitely colored, then a natural question is:
{\it {do there exist some symmetric geometric structures in $G$ with the same color?}}
\\

 A  {\it regular $k$-gon} in $G$ is  a sequence $(g_1,\cdots,g_k)$  of $k$ elements  such that the lengths of $g_1^{-1}g_2$, $g_2^{-1}g_3,$ $\cdots$, $g_{k-1}^{-1}g_k,$ $g_k^{-1}g_1$ are all equal. For integers $a, b, c$, the symbol $a\uparrow b$ means $a^b$; then $a\uparrow b\uparrow c$ stands for $a^{b^{c}}$.We obtained the following theorem.

\begin{theorem}\label{existence of polygon}
For any coloring of  $\mathbb{Z}_2*\mathbb{Z}_2*\mathbb{Z}_2$ with $r$ colors and for any positive integer $k$, there always exists a monochromatic regular $k$-gon in $\mathbb{Z}_2*\mathbb{Z}_2*\mathbb{Z}_2$, the edge length of which is no more than
\begin{equation*}
\left\{
\begin{array}{cl}
4\underbrace{\uparrow\cdots\uparrow}_{2^{r-2}} 4\uparrow d,& k  \text{ is even},\\
4\underbrace{\uparrow\dotsb\uparrow}_{2^{r-2}-1}64\uparrow 4\underbrace{\uparrow\dotsb\uparrow}_{2^{r-3}-1} 64\uparrow
\dotsb \uparrow 64\uparrow d,& k \text{ is odd}.
\end{array}
\right.
\end{equation*}
\end{theorem}

\noindent{\bf Remark.} The main idea of the proof is to find some special patterns with the same color in the Cayley graph of $\mathbb{Z}_2*\mathbb{Z}_2*\mathbb{Z}_2$. In fact, by a not difficult argument, the existence of such patterns can be deduced from the Furstenberg-Weiss theorem
for Ramsey theory on trees (\cite{FW2}); but the proof of the FW-theorem relies on probability methods which affords no further information on the
edge lengths of the regular polygons; another proof in \cite{J} is based on the Szemer\'edi's theorem. The method we used is entirely combinatorial and the edge lengths are esimated. \\

The paper is organized as follows. In section 2, we give some definitions and notations which will be used throughout the paper.
In section 3, we obtain a Ramsey property of binary trees which will be used to show the existence of monochromatic regular polygon with even number of vertices. In section 3, we deal with ternary trees and get a result available to the case of odd number of vertices. In section 4, we show the main theorem of this paper.

\section{Definitions and Notations}
Given a group $G$ and a finite generating set $S$, the\textit{\textbf{ Cayley graph}} $\Gamma_{G,S}=(V,E)$ of $G$ with respect to $S$ is defined as follows:
\begin{itemize}
  \item The set of vertices consists of all elements of $G$, i.e., $V=G$.
  \item There is an edge between two vertices $x$ and $y$ if and only if there is some $s\in S$ such that $y=sx$.
\end{itemize}
Given an element $g\in G$, its {\textit\textbf word length} $|g|$ with respect to $S$ is defined to be the shortest length of a word
$w$  over $S$ whose evaluation is equal to $g$. Given two elements $g, h \in G$, the distance $d(g,h)$ in the word metric with respect to $S$
 is defined to be $|g^{-1}h|$.  Then the Cayley graph  of $\mathbb{Z}_2*\mathbb{Z}_2*\mathbb{Z}_2$ is an infinite complete ternary tree denoted by $T_3$,
 and the distance of $g$ and $h$ is the length of the unique path from $g$ to $h$ in $T_3$.
One may consult \cite{JM} for more details about Cayley graph. \\

Let $T$ be a tree and $V(T)$ denote the set of vertices of $T$. For $x, y\in V(T)$, there is a unique path from $x$ to $y$, which is denoted by $xTy$. If we choose one vertex as special, such a vertex is then called the \textit{\textbf{root}} of $T$ and $T$ is called a \textit{\textbf{rooted tree}}.  For any $x\in V(T)$, we  define the\textit{\textbf{ height}} of $x$ as the length of the path from the root to $x$ in the sense of the number of the edges on the path.  The vertices of height of $k$ form the $k^{\text{th}}$-\textit{\textbf{level}} of $T$.  \\

We can define a tree-order on the vertices of a rooted tree $T$ with root $v$. If $y\in vTx$, then we write $x\leq y$. We shall think of this ordering as expressing height: if $x<y$ we say that $x$ lies \textit{\textbf{below}} $y$ and we also say $x$ is a \textit{\textbf{descendant}} of $y$.  More precisely, we say $x$ is a $k^{\text{th}}$-descendent of $y$, if $x$ lies below $y$ and the the length of the unique path from $x$ to $y$ is $k$.  We say  $x$ is a \textit{\textbf{child}}  of $y$ if  $x$ is a $1^{\text{st}}$-descendent of $y$.  \\
\begin{figure}[htp]
  \centering
  \includegraphics[width=13cm]{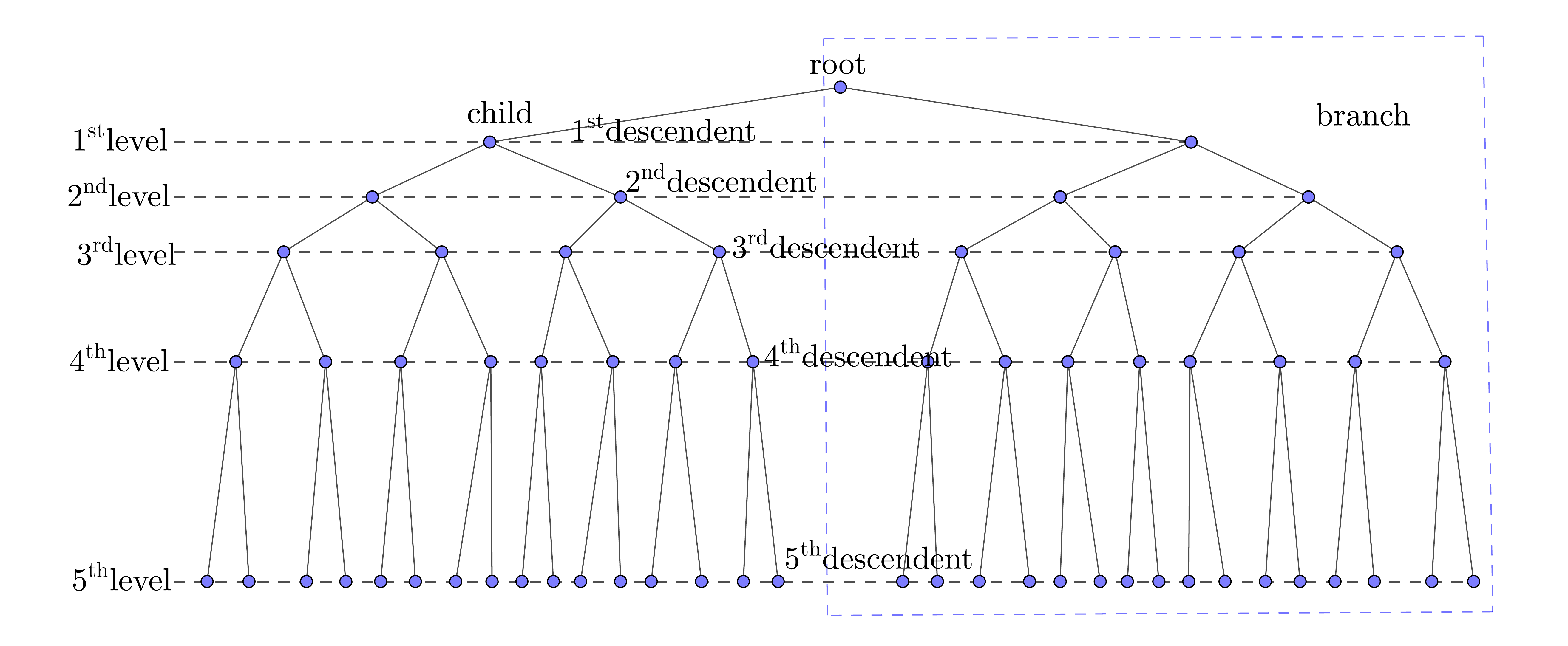}\\
  \caption{A complete binary tree of depth $6$, i.e. $T_6^2$.}\label{binarytree0color}
\end{figure}

For a positive integer $d$, let $T_d^2$ and  $T^3_{d}$  denote the complete binary and ternary tree of depth $d-1$ respectively.\\

Let $T_1$ and $T_2$ be rooted trees. We say there is a\textit{\textbf{ replica}} of $T_1$ in $T_2$ if there is an injective map $\varphi: V(T_1)\longrightarrow V(T_2)$ satisfying:
\begin{itemize}
  \item [(i)]  If $x,y\in V(T_1)$ are in the same level of $T_1$, then $\varphi(x)$ and $\varphi(y)$ are in the same level  of $T_2$.
  \item [(ii)] If $y$ and $z$ are  the descendent of  $x$ in $T_1$, then $\varphi(x)$ and $\varphi(y)$ are  the descendent of $\varphi(z)$ in $T_2$.
\end{itemize}
We call the tree induced by $\varphi(T_1)$ in $T_2$ a replica of $T_1$ in $T_2$ under $\varphi$.\\
\begin{figure}[htp]
  \centering
  \includegraphics[width=15cm]{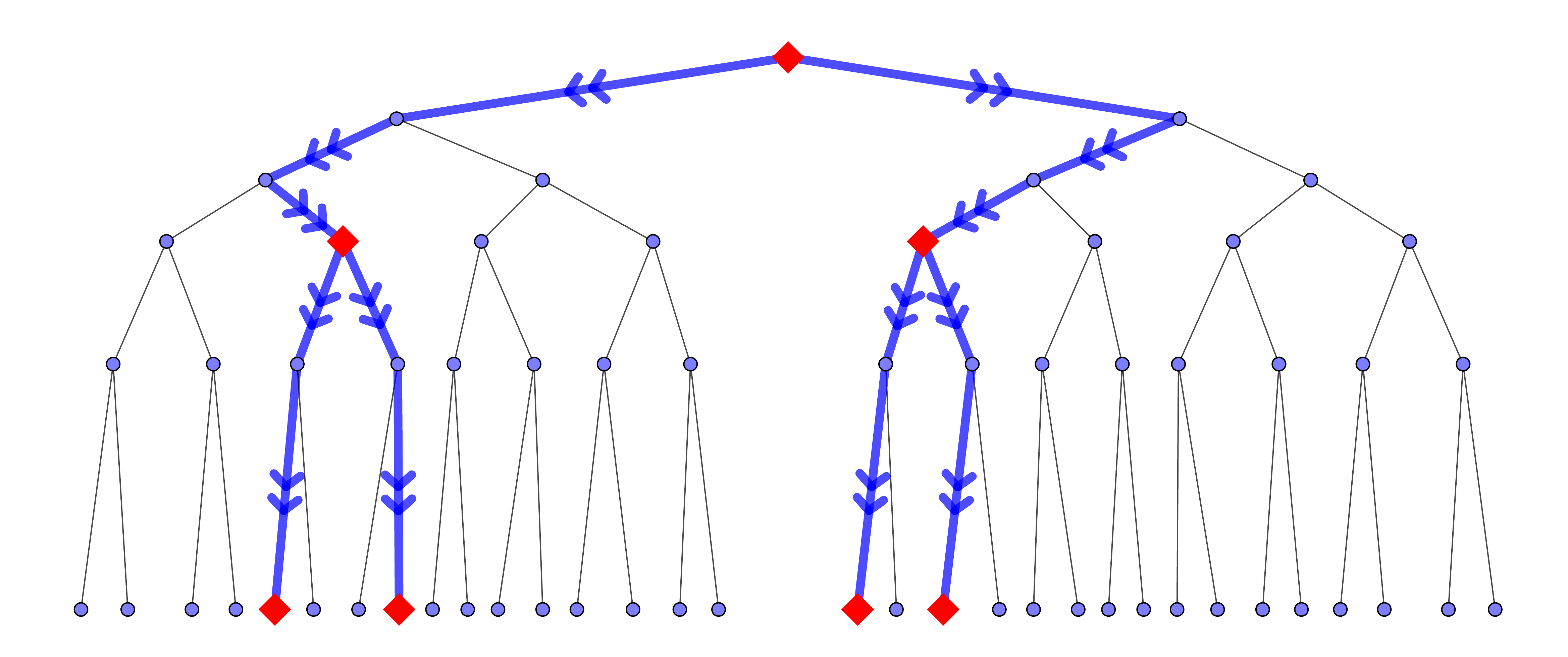}\\
  \caption{A replica of $T_3^2$ in $T_6^2$.}\label{replica}
\end{figure}

A\textit{\textbf{ monochromatic replica of $3$-claw}} of depth $d$ is the set of $2^{d-1}+1$ vertices $\{x_1,\cdots, x_{2^{d-2}}$; $ y_1,\cdots,y_{2^{d-2}};z\}$  satisfying
\begin{itemize}
  \item These vertices are colored with the same color.
  \item There is a vertex $v$ such that all vertices have equal distance to $v$ and $\{x_1,\cdots, x_{2^{d-2}}\}$, $\{y_1,\cdots,y_{2^{d-2}}\}$ and $\{z\}$ lie on three different branches of $v$.
\end{itemize}

\section{A monochromatic replica of $T_d^2$ in $T_n^2$}

\begin{theorem}\label{binary tree}
For any  positive integers $r$ and $d$, there exists a positive integer $N=N(r,d)$ such that for any $n\geq N$, there is a monochromatic replica  of $T_d^2$ in $T_n^2$ for any coloring of the vertices of $T_n^2$ with $r$ colors. Moreover, $N\leq 4\underbrace{\uparrow\cdots\uparrow}_{2^{r-2}} 4\uparrow d$.
\end{theorem}

\begin{proof}  We prove the theorem by induction on $r$. For the case that $r=1$ there is nothing to prove. Suppose the  color set is  $\{1,2,\cdots,r\}$.\\

\noindent\textbf{Case1.} Let $r=2$.  We show the theorem by induction on $d$ in this case. \\

It is trivial when $d=1$. So we assume that $d\geq 2$ and there is a positive integer $N_0$  for which we can find  a monochromatic replica of $T_{d-1}^2$ in $T_{n}^2$  for any  $n\geq N_0$. \\

Let $N=N(2,d)=N_0+d2^d $ and we will show that there is a monochromatic replica  of $T_d^2$ in $T_n^2$ for any $n\geq N$. Now we fix  $n\geq N$ and there is a monochromatic replica of $T_{d-1}^2$ in $T_{n}^2$ associated with a map $\varphi: V(T_{d-1}^2)\longrightarrow V(T_{n}^2) $. Let the image of the $(d-1)^{\text{th}}$-level of $T_{d-1}^2 $ in $T_n^2$  be $v_1,v_2,\cdots, v_{2^{d-1}}$, moreover, which are in the $N_0^{\text{th}}$-level of $T_n^2$ and we may assume that they are colored with $1$. \\

 If all children of $v_1,v_2,\cdots, v_{2^{d-1}}$  are also colored with $1$, then we obtain a monochromatic replica of $T_{d}^2$ by adding all children to the replica of $T_{d-1}^2$ in $T_n^2$. Moreover, for any $k\in\{1,2,\cdots, d2^d\}$, if there is at least one  $k^{\text{th}}$-descendent of each $v_i$  colored with $1$, then we can also find a monochromatic replica of $T_{d}^2$ in $T^2_n$ by adding these descendent. Therefore, the remaining case is that for any $k\in\{1,2,\cdots, d2^d\}$, there exists at least one $i\in \{1,2,\cdots, 2^{d-1}\}$ such that the $k^{\text{th}}$-descendent of $v_i$ are all colored with $2$. But then there is some $v_j$ and $1\leq k_1<k_2<\cdots<k_d\leq d2^d $ such that the $k_1^{\text{th}},k_2^{\text{th}},\cdots,k_d^{\text{th}}$-descendent of such $v_j$ are all colored with $2$.  We can choose a point from the $k_1^{\text{th}}$-descendent of $v_j$ denoted by $x$ and all descendent of $x$ in $(N_0+k_2),\cdots,(N_0+k_d)$-levels of $T_n^2$ then they will form a monochromatic replica of $T_d^2$ in $T_n^2$.

\begin{figure}[htp]
  \centering
  \includegraphics[width=8cm]{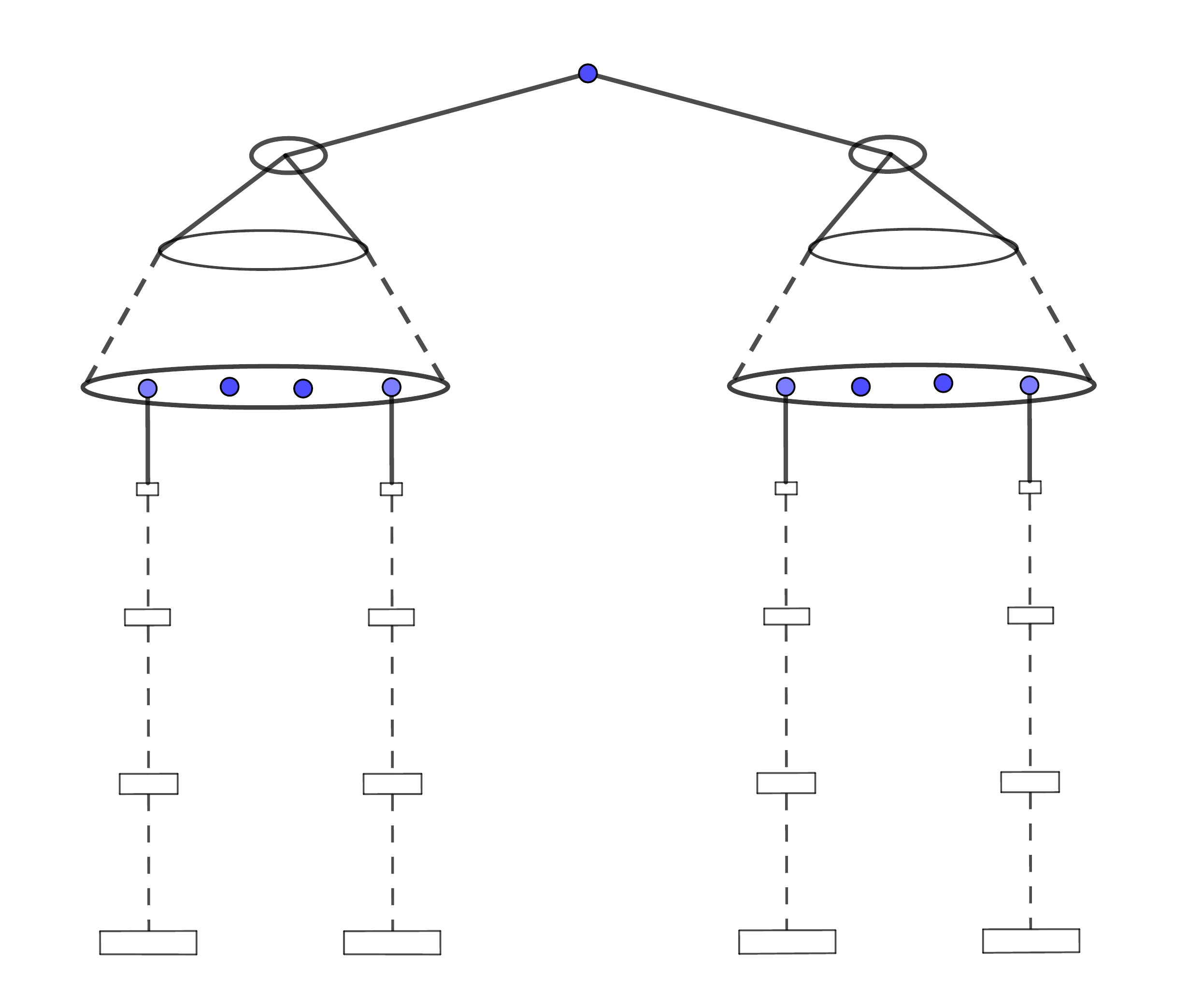}\\
  \caption{Finding monochromatic replica of $T_d$ in $T_n$.}\label{binarytreewith2color}
\end{figure}

 \noindent\textbf{Case2.} Let $r\ge 3$. Suppose that there is a positive integer $N(r-1,d)$ such that, for any $n\geq N(r-1,d)$, there is a monochromatic replica  of $T_d^2$ in $T_n^2$ for any coloring of the vertices of $T_n^2$ with $(r-1)$ colors. \\

 Let $N=N(r,d)=\max\{N(r-1, N(r-1,d)), N(2, N(r-1,d))\}$ and $n\geq N$.\\

  If we substitute the color of $2,3,\cdots, d$ by the color $0$, then $T_n^2$ is color is colored with two colors and we can find a monochromatic replica of $T_{ N(2, N(r-1,d))}^2$ in $T^2_n$. If this monochromatic replica is colored with $1$ , then it is also a monochromatic replica of $T_{ N(2, N(r-1,d))}^2$ in $T^2_n$ and we can find  a monochromatic replica of $T_{ d}^2$ in $T^2_n$ . Otherwise, it is colored with $0$, then it is color with $k-1$ colors and  there is a monochromatic replica of $T_{N(r-1, N(r-1,d))}^2$ in $T^2_n$ by Case1. Consequently, we also can find a monochromatic replica of $T_{d}^2$ in $T^2_n$ by induction hypothesis.\\

  Finally, let's conclude the estimate of $N$. From the above discussion, we know that
  \begin{displaymath}
  N(2,1)=1,~~N(2,d)\leq N(2,d-1)+d2^d.
  \end{displaymath}
  Thus
  \begin{displaymath}
  N(2,d)\leq d2^{d+1}\leq 4^d.
  \end{displaymath}
  Moreover,
  \begin{eqnarray*}
  N(r,d)&=&\max\{N(r-1, N(r-1,d)), N(2, N(r-1,d))\}\\
  &\leq& \max\{4\underbrace{\uparrow\cdots\uparrow}_{2^{r-2}} 4\uparrow d,  ~~~4\underbrace{\uparrow\cdots\uparrow}_{r-2} 4\uparrow d   \}\\
  &=& 4\underbrace{\uparrow\cdots\uparrow}_{2^{r-2}} 4\uparrow d.
  \end{eqnarray*}
   Hence we complete the proof.
\end{proof}

\section{A monochromatic replica of $3$-claw in $T_n^3$}

\begin{lemma}\label{ternary tree in binary tree}
For any positive integer $n$, there is a replica of $T_{\lfloor\frac{n}{2}\rfloor}^3$ in $T_n^2$.
\end{lemma}

\begin{proof}
The result is obvious by Figure \ref{ternary tree in binary tree}.
\end{proof}
\begin{figure}[htp]
  \centering
  \includegraphics[width=15cm]{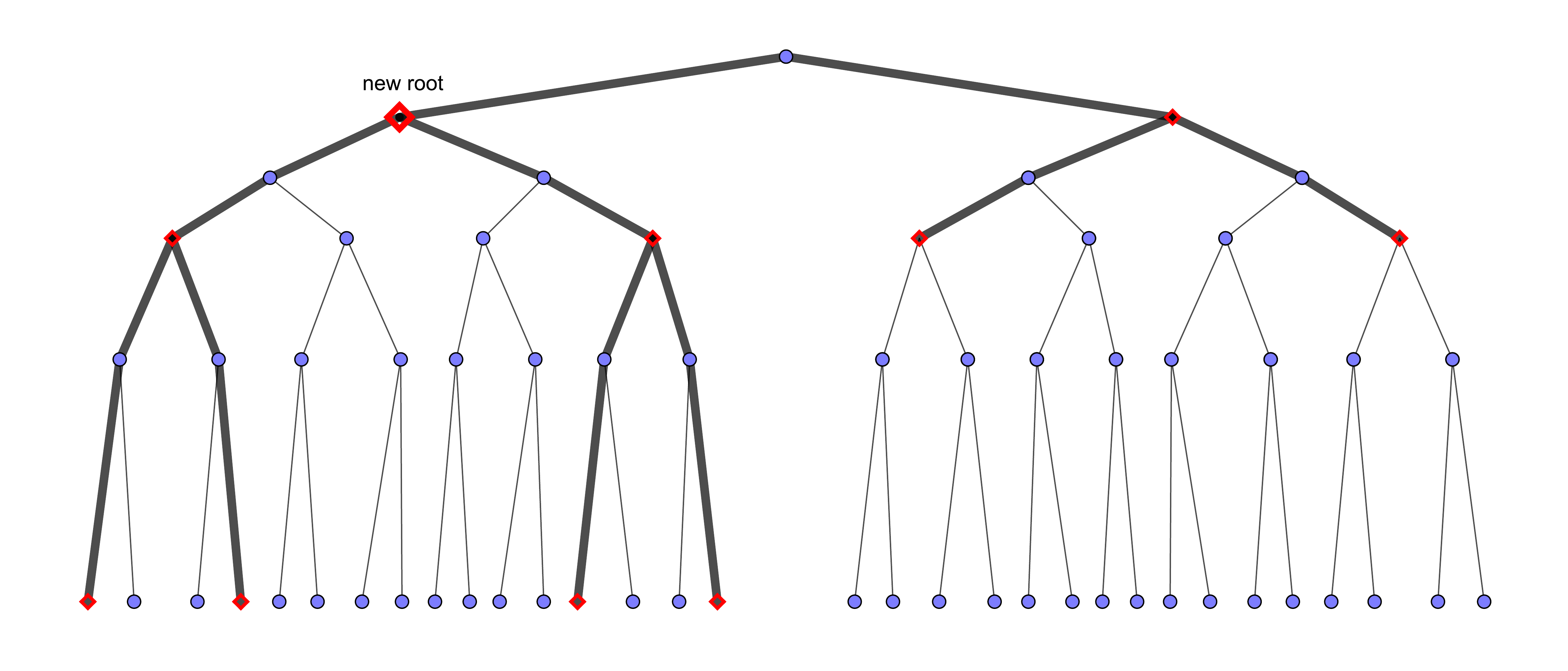}\\
  \caption{Finding $T_{\lfloor\frac{n}{2}\rfloor}^3$ in $T_n^2$.}\label{ternary in binary tree}
\end{figure}

\begin{theorem}\label{3claw}
For any  positive integers $r$ and $d$, there exists a positive integer $M=M(r,d)$ such that for any $n\geq M$, there is a monochromatic replica  of $3$-claw of depth $d$ in $T_n^3$ for any coloring of the vertices of $T_n^3$ with $r$ colors. Moreover
\begin{displaymath}
M(r,d)\leq 4\underbrace{\uparrow\dotsb\uparrow}_{2^{r-2}-1}64\uparrow 4\underbrace{\uparrow\dotsb\uparrow}_{2^{r-3}-1} 64\uparrow
\dotsb \uparrow 64\uparrow d.
\end{displaymath}
\end{theorem}

\begin{proof}  We prove the theorem by induction on $r$. For the case that $r=1$ there is nothing to prove. Suppose the  color set is  $\{1,2,\cdots,r\}$.\\

 Firstly, by Theorem \ref{binary tree}, for any positive integer $\overline{d}$, there is a positive integer $N(r,\overline{d})$ such that there is a monochromatic replica of  $T_{\overline{d}}^2$ in $T_n^3$ for any $n\geq N(r,\overline{d})$. We may assume that this monochromatic replica is colored with $1$. We let the root, saying $v_0$, of this $T_{\overline{d}}^2$ be a new root of $T_n^3$ and we  denote this new rooted tree by $T$. Moreover, we assume that the $k^{\text{th}}$-level of $T_{\overline{d}}^2$ is mapped into the ${h_k}^{\text{th}}$-level of  $T$ for $k\in \{1,2,\cdots,\overline{d}\}$. We call the branches of $T$  which the $T_{\overline{d}}^2$ is imbedded into  branch$1$ and branch$2$ and the another one left branch$3$.\\

 Let $\tilde{d}$ be a positive integer and $\overline{d}\geq 3\tilde{d}$. Now for $k\in \{0, 1,\cdots, 2\tilde{d}-1\}$, we choose a  vertex of the $k^{\text{th}}$-descendent of $v_0$ both in branch$1$ and  in the replica of the $T_{\overline{d}}^2$ that we have chosen.  We denote this vertex by $v_k$. Then there are at least $2^{\overline{d}-k-2}\geq 2^{\tilde{d}-1}$ descendent of $v_k$ colored with $1$ in  ${h_k}^{\text{th}}$-level.\\

\noindent\textbf{Case1.} Let $r=2$ and  $ \tilde{d}=d, ~\overline{d}\geq 3d$.\\

 If there is a vertex, saying $x$, colored with $1$ in the ${h_{\overline{d}}}^{\text{th}}$-level of branch$3$, then there there is a monochromatic replica $3$-claw of depth $d$ in $T_n^3$ by choosing $x$ and $v_0$ associated with its $2^{d-1}$ descendent colored with $1$ in the ${h_{\overline{d}}}^{\text{th}}$-level of branch$1$ and branch$2$. Similarly,  for $k\in \{1,2,\cdots, 2\tilde{d}-1\}$,  if there is a vertex, saying $x_k$, colored with $1$ in the ${(h_{\overline{d}}-2k)}^{\text{th}}$-level of branch$3$, then there is a monochromatic replica $3$-claw of depth $d$ in $T_n^3$ by choosing $x_k$ and $v_k$ associated with its $2^{d-1}$ descendent colored with $1$ in the ${h_{\overline{d}}}^{\text{th}}$-level of branch$1$ .\\

From the above discussion, we may assume all vertices in the ${(h_{\overline{d}}-2k)}^{\text{th}}$-level of branch$3$ are colored with $2$ for any $k\in\{0,1,\cdots 2d-1\}$. But a vertex in in the ${(h_{\overline{d}}-2(2d-1))}^{\text{th}}$-level of branch$3$ and its all descendent in ${(h_{\overline{d}}-2(2d-2))}, {(h_{\overline{d}}-2(2d-3))},\cdots, {h_{\overline{d}}}^{\text{th}}$-level form a monochromatic $T_{2d}^2$. Subsequently, by Lemma \ref{ternary tree in binary tree}, we can find a monochromatic replica of $T_d^3$ which is also true for $3$-claw of depth $d$.\\

Hence we compete the proof in the case that $r=2$ by choosing $M=M(d,2)=N(r, 3d)$.\\

\begin{figure}[htp]
  \centering
  \includegraphics[width=12cm]{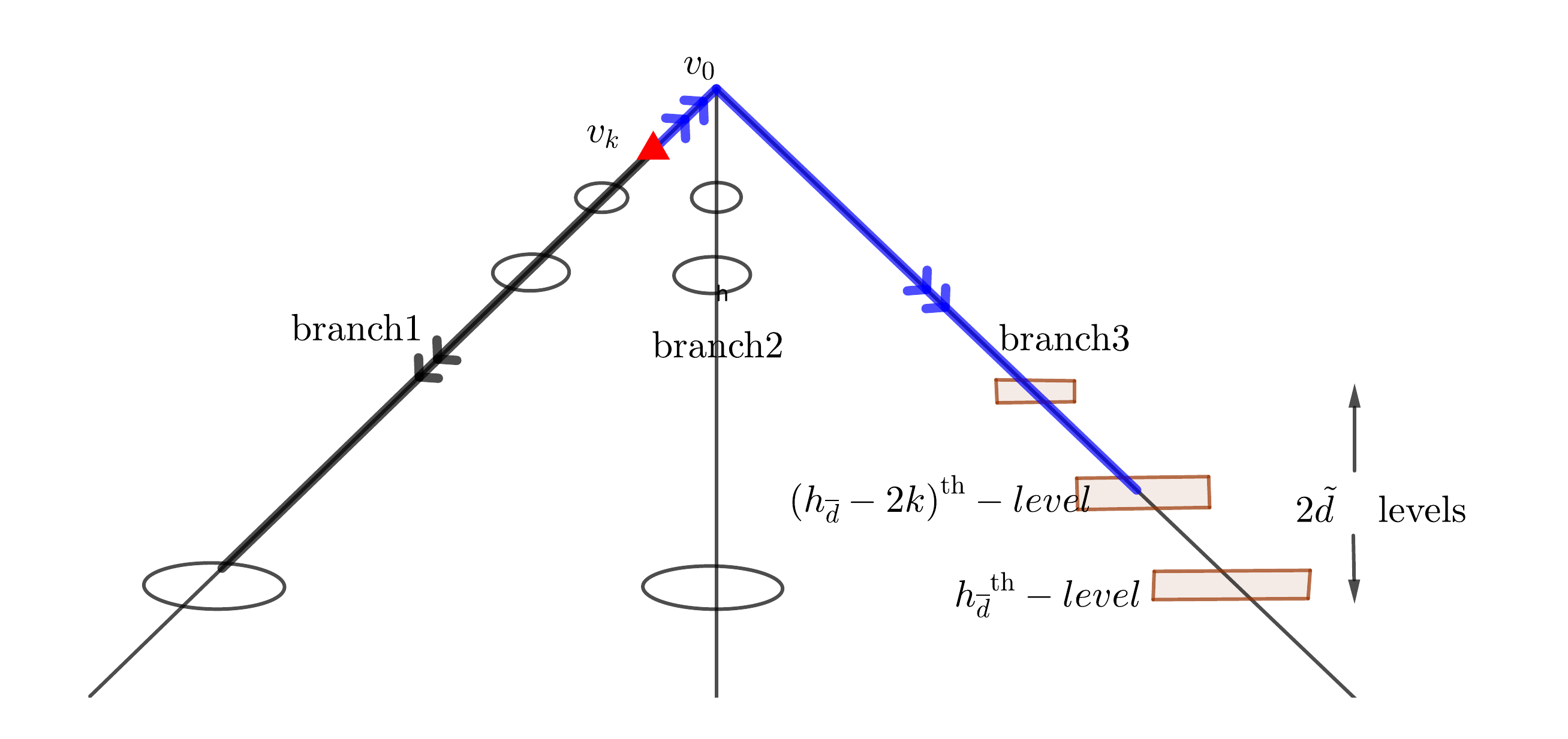}\\
  \caption{Finding $3$-claw in $T_n^3$.}\label{clawinternary}
\end{figure}

\noindent\textbf{Case2.} Let $r\ge 3$ and $\tilde{d}=M(r-1, d),  \overline{d}\geq \tilde{d}+d $.\\

By similar discussion as in Case1,  for $k\in \{1,2,\cdots, 2\tilde{d}-1\}$,  if there is a vertex, saying $x_k$, colored with $1$ in the ${(h_{\overline{d}}-2k)}^{\text{th}}$-level of branch$3$, then there is a monochromatic replica $3$-claw of depth $d$ in $T_n^3$ by choosing $x_k$ and $v_k$ associated with its $2^{d-1}$ descendent colored with $1$ in the ${h_{\overline{d}}}^{\text{th}}$-level of branch$1$. Therefore, we may assume all vertices in the ${(h_{\overline{d}}-2k)}^{\text{th}}$-level of branch$3$ are colored with $r-1$ colors $\{2,\cdots, r\}$ for any $k\in\{0,1,\cdots 2d-1\}$. But a vertex in in the ${(h_{\overline{d}}-2(2\tilde{d}-1))}^{\text{th}}$-level of branch$3$ and its all descendent in ${(h_{\overline{d}}-2(2\tilde{d}-2))}, {(h_{\overline{d}}-2(2\tilde{d}-3))},\cdots,$ $ {h_{\overline{d}}}^{\text{th}}$-level form a monochromatic $T_{2\tilde{d}}^2$. Then we can find a monochromatic replica of $T_{\tilde{d}}^3$ by  Lemma \ref{ternary tree in binary tree}.\\

Since we have chosen $\tilde{d}=M(r-1, d)$, by induction hypothesis, we can find a a monochromatic $3$-claw of depth $d$ .\\

Finally, we estimate the bound the $M(r,d)$. It is clear from the paragraph that
\begin{displaymath}
M(2,d)\leq N(2,3d)\leq 6d8^d,
\end{displaymath}
and
\begin{displaymath}
M(r,d)\leq N(r, 3 M(r-1,d)).
\end{displaymath}
So
\begin{displaymath}
M(r,d)\leq 4\underbrace{\uparrow\dotsb\uparrow}_{2^{r-2}-1}64\uparrow 4\underbrace{\uparrow\dotsb\uparrow}_{2^{r-3}-1} 64\uparrow
\dotsb \uparrow 64\uparrow d.
\end{displaymath}

 Hence we complete the proof.

\end{proof}

\section{The existence of a monochromatic regular polygon}

\noindent\textit{Proof of Theorem \ref{existence of polygon}}.\\

\noindent \textbf{Case1} $k$ is an even number.\\

By Theorem \ref{binary tree}, there is a monochromatic replica of $T_d^2$ in $T_3$.  Choose $d$ such that $2^{d-1}\geq k$. In this replica  we can  choose $\frac{k}{2}$ vertices in the outermost level from each branch. More precisely, let the root of the replica is $v$ and choose $\{x_1,\cdots,x_{\frac{k}{2}}\}$ in one branch and $\{y_1,\cdots,y_{\frac{k}{2}}\}$ in the other. Suppose $x_i$ and $y_j$ corresponds to the elements $g_{2i-1}$ and $g_{2j}$ in $\mathbb{Z}_2*\mathbb{Z}_2*\mathbb{Z}_2$ respectively. Then $(g_1,\cdots,g_k)$ forms a polygon, since the length of $g_i^{-1}g_{i+1}$ equals to the length of the path between the vertices in the Cayley graph corresponding to $g_i$ and $g_{i+1}$.\\

\begin{figure}[htp]
  \centering
  \includegraphics[width=15cm]{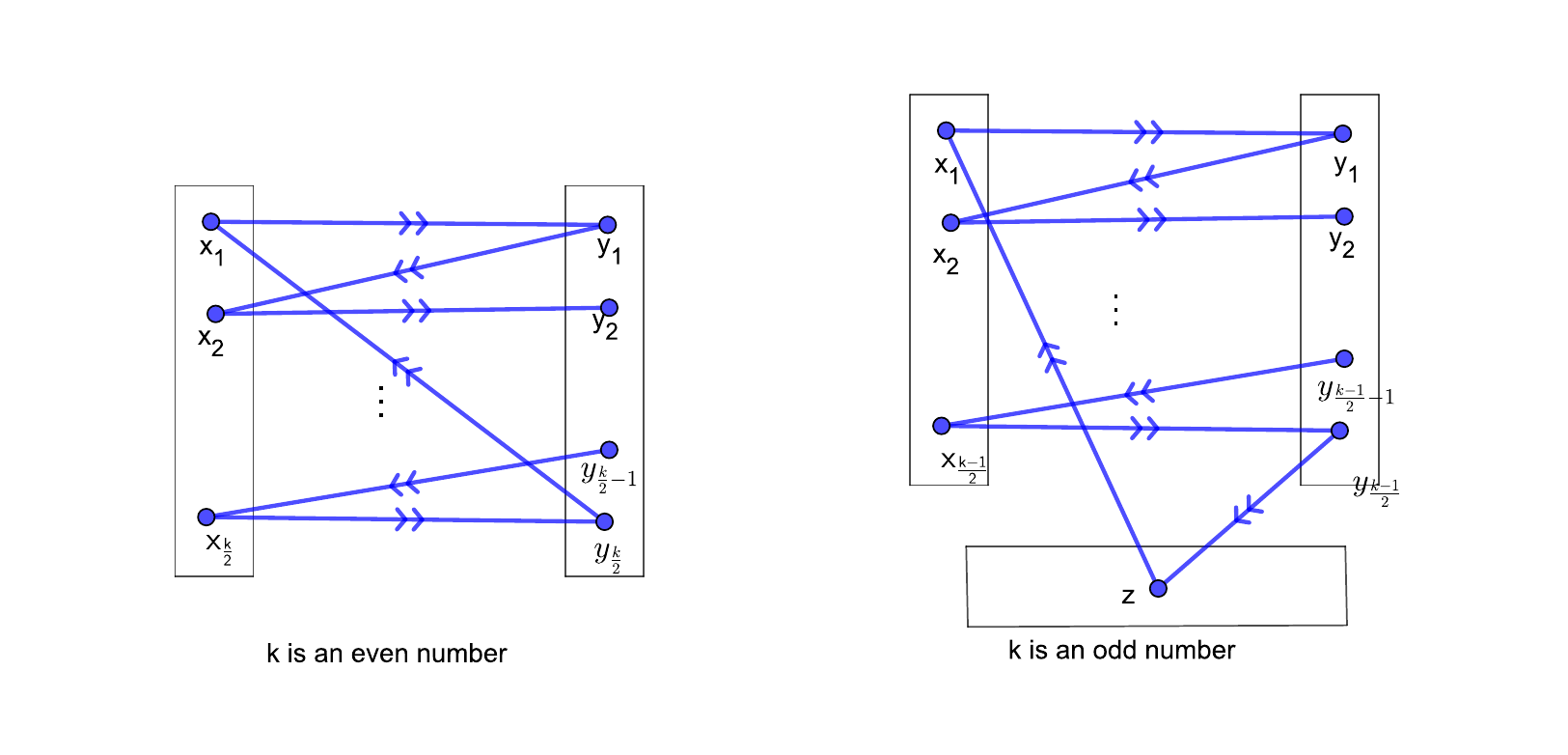}\\
  \caption{Existence of regular polygon.}
\end{figure}

\noindent \textbf{Case2} $k$ is an odd number.\\

By Theorem \ref{3claw}, there is a monochromatic replica of $3$-claw of depth $d$ in $T_3$.  Choose $d$ such that $2^{d-1}\geq k$. In this replica  we can  choose $\frac{k-1}{2}$ vertices in the outermost level from each branch. More precisely, let the root of the replica is $v$ and choose $\{x_1,\cdots,x_{\frac{k-1}{2}}\}$ in one branch and $\{y_1,\cdots,y_{\frac{k-1}{2}}\}$ in the other and $z$ in another one. Suppose $x_i$, $y_j$ ad $z$ corresponds to the elements $g_{2i-1}$, $g_{2j}$  and $g_k$ in $\mathbb{Z}_2*\mathbb{Z}_2*\mathbb{Z}_2$ respectively. Then $(g_1,\cdots,g_k)$ forms a polygon.\\
\rightline{$\Box$}

\addcontentsline{toc}{section}{Reference}
\bibliography{bib}

\vspace{5mm}

\noindent Enhui Shi, School of mathematical and sciences, Soochow University, Suzhou, 215006, P.R. China (ehshi@suda.edu.cn)\\

\noindent Hui Xu, School of mathematical and sciences, Soochow University, Suzhou, 215006, P.R. China (mathegoer@163.com)

\end{document}